\documentclass[11pt]{article}
\usepackage{latexsym,amsfonts,amsthm,amsmath,epsfig,amssymb,cases}
\usepackage{mathrsfs}
\usepackage{enumerate}

\newtheorem{thm}{Theorem}[section]

\newtheorem{prob}{Problem}[section]
\newtheorem{assumption}{Assumption}[section]
\newtheorem{remark}{Remark}[section]
\newtheorem{prop}{Proposition}[section]
\newtheorem{defn}{Definition}[section]
\newtheorem{example}{Example}[section]
\newcounter{nextauthor}
\def\mathrm{\mbox}

\numberwithin{remark}{section}

\begin{document}
\title{{\Large \bf A stochastic optimal control problem governed by SPDEs via a spatial-temporal interaction operator}\thanks{This work was supported by the National Natural Science Foundation of China (11471230, 11671282).}}
\author{Zhun Gou, Nan-jing Huang\footnote{Corresponding author.  E-mail addresses: nanjinghuang@hotmail.com; njhuang@scu.edu.cn}, Ming-hui Wang and Yao-jia Zhang\\
{\small\it Department of Mathematics, Sichuan University, Chengdu, Sichuan 610064, P.R. China}}
\date{}
\maketitle
\begin{center}
\begin{minipage}{5.5in}
\noindent{\bf Abstract.} In this paper,  we first introduce a new spatial-temporal interaction operator to describe the space-time dependent phenomena. Then we consider the stochastic optimal control of a new system governed by a stochastic partial differential equation with the spatial-temporal interaction operator. To solve such a stochastic optimal control problem, we derive an adjoint backward stochastic partial differential equation with spatial-temporal dependence by defining a Hamiltonian functional, and give both the sufficient and necessary (Pontryagin-Bismut-Bensoussan type) maximum principles.
Moreover, the existence and uniqueness of solutions are proved for the corresponding adjoint backward stochastic partial differential equations. Finally, our results are applied to study the population growth problems with the space-time dependent phenomena.
\end{minipage}
\end{center}

\section{Introduction}
\paragraph{}
In last decades, many scholars have focused on the topic of stochastic partial differential equations (SPDEs), which has many real world applications \cite{holden1996stochastic, liu2016analysis, ma1997Adapted, mijena2016intermittence}. In this paper, we consider a stochastic optimal control problem governed by a new SPDE with a spatial-temporal interaction operator which can be used to describe the space-time dependent phenomena appearing in population growth problems. To explain the motivations of our work, we first recall some recent works concerning on the stochastic optimal control problems governed by SPDEs.

In 2005, {\O}ksendal \cite{Oksendal2005optimal} studied the stochastic optimal control problem governed by the SPDE, proved a sufficient maximum principle for the problem, and  applied the results to solve the optimal harvesting problem described by the SPDE without the time delay.
However, there are many models with past dependence in realistic world, in which the optimal control problems governed by some dynamic systems with time delays have more practical applications. For example, for biological reasons, time delays occur naturally in population dynamic models \cite{Mohammed1998Stochastic, Oksendal2011optimal}. Therefore, when dealing with optimal harvesting problems of biological systems, one can be led to the optimal control problems of the systems with time delays. Motivated by this fact, {\O}ksendal  et al. \cite{oksendal2012Optimal} investigated the stochastic optimal control problem governed by the delay stochastic partial differential equation (DSPDE), established both sufficient and necessary stochastic maximum principles for this problem, and illustrated their results by an application to the optimal harvesting problem from a biological system. Besides, we note that another area of applications is mathematical finance, where time delays in the dynamics can represent memory or inertia in the financial system (see, for example, \cite{agram2019stochastic}). Some other applications, we refer the reader to \cite{basse2018multivariate, Gopalsamy2013Stability, Kocic2010Generalized, meng2015optimal, Mokkedem2019Optimal} and the references therein.

On the other hand, it is equally important to study the stochastic optimal control problem governed by dynamic system with the spatial dependence because it also has many applications in real problems such as the harvesting problems of biological systems \cite{hening2018stochastic, Schreiber2009Invasion}. To deal with the problems, Agram et al. \cite{agram2019spdes} introduced the space-averaging operator and considered a system of the SPDE with this type operator. Then they proved both sufficient and necessary stochastic maximum principles for the problem governed by such an SPDE and applied the results to solve the optimal harvesting problem for a population growth system in an environment with space-mean interactions. Following \cite{agram2019spdes}, Agram  et al. \cite{Agram2019Singular} also solved a singular control problem of optimal harvesting from a fish population, of which the density is driven by the SPDE with the space-averaging operator. For some related works concerned with the optimal control problems for SPDEs, we refer the reader to \cite{bensoussan2004stochastic, Da2014Stochastic,  Dumitrescu2018Stochastic,  Fuhrman2016Stochastic, Hu1990Maximum, lu2015Stochastic, Wu2019Boundary}.

Now, a natural question arises: can we describe both of the past dependence and the space-mean dependence in the same framework? Moreover, the question is generalized as follows: can we describe the spatial-temporal dependence of the state in the stochastic system? To this end, we construct the spatial-temporal interaction operator. Then, we consider the stochastic optimal control problem in which the state is governed by the new system of the SPDE with this operator in the filtered probability space $(\Omega,\mathscr{F},\mathscr{F}_t,\mathbb{P})$ satisfying the usual hypothesis. This system takes the following form:
\begin{equation}\label{SDE}
\begin{cases}
dX(t,x)&=\left(A_xX(t,x)+b(t,x)-u(t,x)\right)dt+\sigma(t,x)dB_t+\int_{\mathbb{R}_0}\gamma(t,x,\zeta)\widetilde{N}(dt,d\zeta),\\
&\qquad\qquad\qquad\qquad\qquad\qquad\qquad\qquad\qquad\qquad\;(t,x)\in[0,T]\times D;\\
X(t,x)&=\xi(t,x), \qquad\qquad\qquad\qquad\qquad\qquad\,\qquad\qquad (t,x)\in(0,T]\times \partial D;\\
X(t,x)&=\eta(t,x), \qquad\qquad\qquad\qquad\qquad\qquad\,\qquad\qquad (t,x)\in[-\delta,0]\times \overline{D};\\
u(t,x)&=\beta(t,x), \qquad\qquad\qquad\qquad\qquad\qquad\qquad\,\qquad (t,x)\in[-\delta,0]\times \overline{D},
\end{cases}
\end{equation}
where $dX(t,x)$ is the differential with respect to $t$, $u(t,x)$ is the control process and $D\subset \mathbb{R}^d$ is an open set with $C^1$ boundary $\partial D$. Moreover, $\overline{D}=D\bigcup \partial D$. We have used simplified notations in equation \eqref{SDE} such that
\begin{align*}
b(t,x)&=b(t,x,X(t,x),\overline{X}(t,x),u(t,x),\overline{u}(t,x)),\\
\sigma(t,x)&=\sigma(t,x,X(t,x),\overline{X}(t,x),u(t,x),\overline{u}(t,x)),\\
\gamma(t,x,\zeta)&=b(t,x,\zeta,X(t,x)),\overline{X}(t,x),u(t,x),\overline{u}(t,x)).
\end{align*}
where $\overline{X}(t,x)$ denotes the space-time dependent density.

Consequently, we focus on the study of the following stochastic optimal control problem which captures the spatial-temporal dependence.
\begin{prob}\label{problem}
Suppose that the performance functional associated to the control $u\in \mathcal{U}^{ad}$ takes the form
\begin{equation*}
J(u)=\mathbb{E}\left[\int_0^T\int_Df(t,x,X(t,x),\overline{X}(t,x),u(t,x),\overline{u}(t,x))dxdt+\int_Dg(x,X(T,x))dx\right],
\end{equation*}
where $X(t,x)$ is described by \eqref{SDE}, $f$ and $g$ are two given functions satisfying some mild conditions, and $\mathcal{U}^{ad}$ is the set of all admissible control processes.
The problem is to find the optimal control $\widehat{u}=\widehat{u}(t,x) \in \mathcal{U}^{ad}$ such that
\begin{equation}\label{prob}
J(\widehat{u})=\sup \limits_{u\in\mathcal{U}^{ad}}J(u).
\end{equation}
\end{prob}

The rest of this paper is structured as follows. The next section introduces some necessary preliminaries including the definition of the spatial-temporal interaction operator, and derives an adjoint backward stochastic partial differential equation (BSPDE) with spatial-temporal dependence by defining a Hamiltonian functional. In Section 3, the sufficient and necessary maximum principles of the related control problem are derived, respectively. In Section 4, the existence and uniqueness of solutions are obtained for the related BSPDE of the control problem with the spatial-temporal interaction operator. Finally, two examples are presented in Section 5 as applications of our main results.

\section{Preliminaries}
In this section, some necessary definitions and propositions are given to state \eqref{SDE} in detail. We also give several examples to show that all these definitions are well-posed.

Now, in \eqref{SDE}, the terms $B_t$ and $\widetilde{N}(dt,d\zeta)$ denote a one-dimensional $\mathcal{F}_t$-adapted Brownian motion and a compensated Poisson random measure, respectively, such that
$$
\widetilde{N}(dt,d\zeta)={N}(dt,d\zeta)-\nu(dt,d\zeta),
$$
where ${N}(dt,d\zeta)$ is a Poisson random measure associated with the one-dimensional $\mathcal{M}_t$-adapted Poisson process $P_N(t)$ defined on $\mathbb{R}_0=\mathbb{R}\setminus \{0\}$ with the characteristic measure $\nu(dt,d\zeta)$. Here, $B_t$ and $P_N(t)$ are mutually independent. Moreover, $\sigma$-algebras $\mathcal{F}=(\mathcal{F}_t)_{t\geq0}$ and $\mathcal{M}=(\mathcal{M}_t)_{t\geq0}$ are right-continuous and increasing. The augmented $\sigma$-algebra $\mathscr{F}_t$ is generated by
$$
\mathscr{F}_t=\sigma\left(\mathcal{F}_t\vee\mathcal{M}_t\right).
$$
We extend $X(t,x)$ to the process on $[0,T]\times \mathbb{R}^d$ by setting
$$
X(t,x)=0,\quad (t,x)\in [-\delta,T]\times \mathbb{R}^d\setminus \overline{D}.
$$

Next, we recall some useful sets and spaces which will be used throughout this paper.
\begin{defn}${}$
\begin{itemize}
\item $H=L^2(D)$ is the set of all Lebesgue measurable functions $f:D\rightarrow\mathbb{R}$ such that
$$
\|f\|_{H}:=\left(\int_{D}|f(x)|^2dx\right)^{\frac{1}{2}}<\infty,\quad x\in D.
$$
In addition, $\langle f(x),g(x)\rangle_H=\int_{D}f(x)g(x)dx$ denotes the inner product in $H$.

\item $\mathcal{R}$ denotes the set of Lebesgue measurable functions $r:\mathbb{R}_0\times D\rightarrow \mathbb{R}$. $L^2_{\nu}(H)$ is the set of all Lebesgue measurable functions $\gamma\in \mathcal{R}$ such that
$$
\|\gamma\|_{L^2_{\nu}(H)}:=\left(\int_{D}\int_{\mathbb{R}_0}|\gamma(x,\zeta)|^2\nu(d\zeta) dx\right)^{\frac{1}{2}}<\infty,\quad x\in D.
$$

\item $H_{T}=L^2_{\mathscr{F}}([0,T]\times \Omega,H)$ is the set of all $\mathscr{F}$-adapted processes $X(t,x)$ such that
$$
\|X(t,x)\|_{H_T}:=\mathbb{E}\left(\int_{D}\int_0^T|X(t,x)|^2dtdx\right)^{\frac{1}{2}}<\infty.
$$

\item $H^{-\delta}_{T}=L^2_{\mathscr{F}}([-\delta,T]\times \Omega,H)$ is the set of all $\mathscr{F}$-adapted processes $X(t,x)$ such that
$$
\|X(t,x)\|_{H^{-\delta}_T}:=\mathbb{E}\left(\int_{D}\int_{^{-\delta}}^T|X(t,x)|^2dtdx\right)^{\frac{1}{2}}<\infty.
$$

\item $V=W^{1,2}(D)$ is a separable Hilbert space (the Sobolev space of order $1$) which is  continuously, densely imbedded in $H$. Consider the topological dual of $V$ as follows:
$$
V\subset H\cong H^{*}\subset V^{*}.
$$
In addition, let $\langle A_xu,u\rangle_{*}$ be the duality product between $V$ and $V^{*}$, and $\|\cdot\|_V$ the norm in the Hilbert space $V$.

\item $\mathcal{U}^{ad}$ is the set of all stochastic processes which take values in a convex subset $\mathcal{U}$ of $\mathbb{R}^d$ and are adapted to a given subfiltration $\mathbb{G}=(\mathcal{G}_t)_{t\geq0}$. Here, $\mathcal{G}_t\subseteq \mathscr{F}_t$ for all $t\geq0$. Moreover, $\mathcal{U}^{ad}$ is called the set of admissible control processes $u$.
\end{itemize}
\end{defn}

\begin{defn}
The adjoint operator $A_x^{*}$ of a linear operator $A_x$ on $C_0^{\infty}(\mathbb{R}^d)$ is defined by
$$
\langle A_x\phi,\psi\rangle_{L^2(\mathbb{R}^d)}=\langle \phi,A_x^{*}\psi\rangle_{L^2(\mathbb{R}^d)},\quad \forall \phi,\psi\in C_0^{\infty}(\mathbb{R}^d).
$$
Here, $\langle \phi_1,\phi_2\rangle_{L^2(\mathbb{R}^d)}=\int_{\mathbb{R}^d}\phi_1(x)\phi_2(x)dx$ is the inner product in $L^2(\mathbb{R}^d)$. If $A_x$ is the second order partial differential operator acting on $x$ given by
$$
A_x\phi=\sum \limits_{i,j=1}^n \alpha_{ij}(x)\frac{\partial^2 \phi}{\partial x_i\partial x_j}+\sum\limits_{i=1}^n\beta_{i}(x)\frac{\partial \phi}{\partial x_i},\quad \forall \phi\in C^2(\mathbb{R}^d),
$$
where $(\alpha_{ij}(x))_{1\leq i,j\leq n}$ is a given nonnegative definite $n\times n$ matrix with entries $\alpha_{ij}(x)\in C^2(D)\bigcap C(\overline{D})$ for all $i,j=1,2,\ldots, n$ and $\beta_{i}(x)\in C^2(D)\bigcap C(\overline{D})$ for all $i=1,2,\ldots, n$,  then it is easy to show that
$$
A_x^{*}\phi=\sum \limits_{i,j=1}^n \frac{\partial^2 }{\partial x_i\partial x_j}(\alpha_{ij}(x)\phi(x))-\sum\limits_{i=1}^n\frac{\partial}{\partial x_i}(\beta_{i}(x)\phi(x)),\quad \forall \phi\in C^2(\mathbb{R}^d).
$$
\end{defn}

We interpret $X(t,x)$ as a weak (variational) solution to \eqref{SDE}, if for $t\in[0,T]$ and all $\phi\in C_0^{\infty}(D)$, the following equation holds.
\begin{align}\label{+1}
\langle X(t,x),\phi\rangle_H=&\langle \beta(0,x),\phi\rangle_H+\int_0^t\langle X(s,x),A_x^{*}\phi\rangle_{*}ds+\int_0^t\langle b(s,X(s,x)),\phi\rangle_Hds\nonumber\\
&+\int_0^t\langle \sigma(s,X(s,x)),\phi\rangle_HdB_s+\int_0^t\int_{\mathbb{R}_0}\langle \gamma(s,X(s,x),\zeta),\phi\rangle_Hd\widetilde{N}(s,\zeta).
\end{align}
In equation \eqref{+1}, these coefficients $b$, $\sigma$ and $\gamma$ are all the simplified notations.

Now, we give the definition of the spatial-temporal interaction operator.
\begin{defn}\label{+3}
 $S$ is said to be a spatial-temporal interaction operator if it takes the following form
\begin{eqnarray}\label{space-averaging}
S(X(t,x))=\int_{R_{\theta}}\int_{t-\delta}^tQ(t,s,x,y)X(s,x+y)dsdy\quad (X(t,x)\in H^{-\delta}_{T}),
\end{eqnarray}
where $Q(t,s,x,y)$ denotes the density function such that
\begin{equation}\label{+2}
\int_{y-R_{\theta}}\int_{s\vee 0}^{(s+\delta)\wedge T}|Q(t,s,x,y-x)|^2 dtdx\leq M.
\end{equation}
Here the set
$$
R_{\theta}=\{y\in \mathbb{R}^d;\|y\|_2<\theta\}
$$
is an open ball of radius $\theta>0$ centered at $0$, where $\|\cdot\|_2$ represents the Euclid norm in $\mathbb{R}^d$.
\end{defn}

\begin{prop}\label{+6}
For any $X(t,x)\in H^{-\delta}_{T}$, one has
\begin{equation}\label{norm of S}
\|S(X(t,x))\|_{H_T}\leq \sqrt{M}\|X(t,x)\|_{H^{-\delta}_{T}}.
\end{equation}
This implies that $S:H^{-\delta}_{T}\rightarrow H_{T}$ is a bounded linear operator.
\end{prop}

\begin{proof}
Applying Cauchy-Schwartz's inequality and Fubini's theorem, we have
\begin{align*}
\|S(X(t,x))\|^2_{H_T}&=\mathbb{E}\left[\int_D\int_0^T\left[\int_{R_{\theta}}\int_{t-\delta}^tQ(t,s,x,y)X(s,x+y)dsdy\right]^2dxdt\right]\\
&\leq \mathbb{E}\left[\int_D\int_0^T\int_{R_{\theta}}\int_{t-\delta}^t|Q(t,s,x,y)|^2|X(s,x+y)|^2 dsdydxdt\right]\\
&=\mathbb{E}\left[\int_D\int_{-\delta}^T\int_{R_{\theta}}\left(\int_{s\vee 0}^{(s+\delta)\wedge T}|Q(t,s,x,y)|^2 dt\right)|X(s,x+y)|^2dydsdx\right]\\
&=\mathbb{E}\left[\int_D\int_{-\delta}^T\int_{x+R_{\theta}}\left(\int_{s\vee 0}^{(s+\delta)\wedge T}|Q(t,s,x,z-x)|^2 dt\right)|X(s,z)|^2dzdsdx\right]\\
&=\mathbb{E}\left[\int_D\int_{-\delta}^T\left(\int_{D\cap(z-R_{\theta})}\int_{s\vee 0}^{(s+\delta)\wedge T}|Q(t,s,x,z-x)|^2 dtdx\right)|X(s,z)|^2ds dz\right]\\
&\leq M\mathbb{E}\left[\int_D\int_{-\delta}^T|X(s,z)|^2dzds\right]=M\|X(t,x)\|_{H^{-\delta}_{T}}^2
\end{align*}
This completes the proof.
\end{proof}

\begin{example}\label{example}
We give examples for spatial-temporal interaction operators in the following three cases, respectively.
\begin{enumerate}[($\romannumeral1$)]
\item If we set
$$
Q_0(t,s,x,y-x)=e^{-\rho(t-s)}e^{\|y\|_2},
$$
where $\rho_1,\rho_2$ are two positive constants, then $Q_0(t,s,x,y-x)$ clearly satisfies condition \eqref{+2} and
$S_0:H^{-\delta}_{T}\rightarrow H_{T}$,
$$
S_0(X(t,x))=\int_{R_{\theta}}\int_{t-\delta}^te^{-\rho_1(t-s)}e^{-\rho_2\|y\|_2}X(s,x+y)dtdx \quad (\forall X(t,x)\in H^{-\delta}_{T})
$$
becomes the spatial-temporal interaction operator. It shows that an increase in distance $\|y\|_2$ or time interval $t-s$ results in a decreasing effect for local population density.

\item When there is no temporal dependence, we set
$S_1:H\rightarrow H:$
$$
S_1(X(t,x))=\int_{R_{\theta}}Q_1(x,y)X(t,x+y)dx \quad (\forall X(t,x)\in H),
$$
where the density function $Q_1(x,y)$ satisfies
$$
\int_{y-R_{\theta}}|Q_1(x,y-x)|^2 dx\leq M.
$$
For $Q_1(x,y)=\frac{1}{V(R_{\theta})}$, where $V(\cdot)$ is the Lebesgue volume in $\mathbb{R}^d$, $S_1$ reduces to the space-averaging operator proposed in \cite{agram2019spdes}.

\item When there is no spatial dependence, we set
$S_2:H^{-\delta}_{T}\rightarrow H_{T}$,
$$
S_2(X(t,x))=\int_{t-\delta}^{t}Q_2(t,s)X(s,x) ds \quad (\forall X(t,x)\in H),
$$
where the density function $Q_2(x,y)$ satisfies
$$
\int_{s\vee 0}^{(s+\delta)\wedge T}|Q_2(t,s)|^2 dt\leq M.
$$
For $Q_2(t,s)=1$, $S_2$ reduces to the well-known moving average operator.
\end{enumerate}
\end{example}

In the sequel, we illustrate the Fr\'{e}chet derivative for spatial-temporal interaction operators.

\begin{defn}
The Fr\'{e}chet derivative $\nabla_{S}F$ of a map $F:H^{-\delta}_{T}\rightarrow H_{T}$ has a dual function if
$$
\mathbb{E}\left[\int_D\int_0^T\langle \nabla_SF,X\rangle(t,x) dxdt\right]=\mathbb{E}\left[\int_D\int_{-\delta}^T\nabla_S^*F(t,x)X(t,x)dxdt\right], \quad \forall X(t,x)\in H^{-\delta}_{T}.
$$
\end{defn}

\begin{example}\label{exam1}
Let $F:H^{-\delta}_{T}\rightarrow H_{T}$ be a given map by setting
$$
F(X)(t,x)=\langle F,X\rangle(t,x)=S(X(t,x))=\int_{R_{\theta}}\int_{t-\delta}^tQ(t,s,x,y)X(s,x+y)dsdy,(t\geq 0)\quad X(t,x)\in H^{-\delta}_{T}.
$$
Since $F$ is linear, for any $X(t,x)\in H^{-\delta}_{T}$, we have
$$
\langle \nabla_SF,\psi\rangle(t,x)=\langle F,\psi\rangle(t,x)=\int_{R_{\theta}}\int_{t-\delta}^tQ(t,s,x,y)X(s,x+y)dsdy
$$
and so
\begin{align*}
&\mathbb{E}\left[\int_D\int_0^T\langle \nabla_SF,\psi\rangle dxdt\right]\\
=&\mathbb{E}\left[\int_D\int_0^T\int_{R_{\theta}}\int_{t-\delta}^tQ(t,s,x,y)X(s,x+y)dsdy dxdt\right]\\
=&\mathbb{E}\left[\int_D\int_{-\delta}^T\int_{R_{\theta}}\left(\int_{s\vee 0}^{(s+\delta)\wedge T}Q(t,s,x,y) dt\right)X(s,x+y)dydxds\right]\\
=&\mathbb{E}\left[\int_D\int_{-\delta}^T\left(\int_{D\cap(z-R_{\theta})}\int_{s\vee 0}^{(s+\delta)\wedge T}Q(t,s,x,z-x) dtdx\right)X(s,z)dsdz\right], \quad \forall \psi\in H.
\end{align*}
This implies that
$$
\nabla^{*}_{S}F(s,z)=\int_{D\cap(z-R_{\theta})}\int_{s\vee 0}^{(s+\delta)\wedge T}Q(t,s,x,z-x) dtdx.
$$
Therefore for $t\in[-\delta,T]$,
$$
\nabla^{*}_{S}F(t,x)=\int_{D\cap(x-R_{\theta})}\int_{t\vee 0}^{(t+\delta)\wedge T}Q(s,t,y,x-y) dsdy
=\int_{D}\int_{t}^{T}Q(s,t,y,x-y)\mathbb{I}_{x-R_{\theta}}(y)\mathbb{I}_{[0,T-\delta]}(t) dsdy.
$$
\end{example}

\begin{remark}
For any $X=X(t,x)\in H$, we set
$$
\overline{X}(t,x)=S(X(t,x)), \quad  \overline{u}(t,x)=S(u(t,x)).
$$
\end{remark}

Now, we introduce these coefficients of SPDE \eqref{SDE} and the functions in Problem \ref{problem} in detail. We assume that all of these are functions in $C^1(H)$ and take the following forms:
\begin{align*}
b(t,x,X,S_X,u,S_u)=&b(t,x,X,S_X,u,S_u,\omega):E\rightarrow \mathbb{R};\\
\sigma(t,x,X,S_X,u,S_u)=&\sigma(t,x,X,S_X,u,S_u,\omega):E\rightarrow \mathbb{R};\\
\gamma(t,x,X,S_X,u,S_u,\zeta)=&\gamma(t,x,X,S_X,u,S_u,\zeta,\omega):E'\rightarrow \mathbb{R};\\
f(t,x,X,S_X,u,S_u)&=f(t,x,X,S_X,u,S_u,\omega):E\rightarrow \mathbb{R};\\
g(x,X(T))&=g(x,X(T),\omega):E''\rightarrow\mathbb{R},
\end{align*}
where
\begin{align*}
E=&[-\delta,T]\times D\times\mathbb{R}\times \mathbb{R}\times \mathcal{U}^{ad}\times \mathbb{R} \times \Omega;\\
E'=&[-\delta,T]\times D\times\mathbb{R}\times \mathbb{R}\times \mathcal{U}^{ad}\times \mathbb{R} \times \mathbb{R}_0 \times \Omega;\\
E''=&D\times  \mathbb{R}\times \Omega.
\end{align*}

Next, we define the related Hamiltonian functional.
\begin{defn}\label{+4}
Define the Hamiltonian functional with respect to the optimal control problem \eqref{prob} by $H:[0,T+\delta]\times D\times \mathbb{R}\times \mathscr{L}(\mathbb{R}^d)\times \mathbb{R} \times\mathcal{U}^{ad}\times \mathbb{R}\times \mathbb{R}\times \mathbb{R}\times \mathcal{R}\times \Omega\rightarrow \mathbb{R}$ as follows:
\begin{align}\label{+14}
H(t,x)&=H(t,x,X,S_X,u,S_u,p,q,r(\cdot))\nonumber\\
&=H(t,x,X,S_X,u,S_u,p,q,r(\cdot),\omega)\nonumber\\
&=f(t,x,X,S_X,u,S_u)+b(t,x,X,S_X,u,S_u)p+\sigma(t,x,X,S_X,u,S_u)q\nonumber\\
&\quad \mbox{}+\int_{\mathbb{R}_0}\gamma(t,x,X,S_X,u,S_u,\zeta)rd\zeta
\end{align}
Moreover, we suppose that functions $b$, $\sigma$, $\gamma$, $f$ and $H$ all admit bounded Fr\'{e}chet derivatives with respect to $X$, $S_X$, $u$ and $S_u$, respectively.
\end{defn}
We associate the following adjoint BSPDE to the Hamiltonian \eqref{+14} in the unknown processes $p(t,x),q(t,x),r(t,x,\cdot)$.
\begin{equation}\label{+5}
\begin{cases}
dp(t,x)&=-\left(\frac{\partial H}{\partial X}(t,x)+A^{*}_xp(t,x)+\mathbb{E}\left[\nabla^*_{S_X}H(t,x)\Big|\mathscr{F}_t\right]\right)dt+q(t,x)dB_t+\int_{\mathbb{R}_0}r(t,x,\zeta)\widetilde{N}(dt,d\zeta),\\
&\quad\qquad\qquad\qquad\qquad\qquad\qquad\qquad\qquad\qquad\qquad\quad\;\;(t,x)\in[0,T]\times D;\\
p(t,x)&=\frac{\partial g}{\partial X}(T,x),\qquad\qquad\qquad\qquad\qquad\qquad\qquad\qquad\;\quad (t,x)\in[T,T+\delta]\times \overline{D};\\
p(t,x)&=0, \qquad\qquad\qquad\qquad\qquad\quad\;\;\qquad\qquad\qquad\qquad\quad (t,x)\in[0,T)\times \partial D;\\
q(t,x)&=0, \qquad\qquad\qquad\qquad\qquad\quad\;\;\qquad\qquad\qquad\qquad\quad (t,x)\in[T,T+\delta]\times \overline{D};\\
r(t,x,\cdot)&=0, \qquad\qquad\qquad\qquad\qquad\quad\;\;\qquad\qquad\qquad\qquad\quad (t,x)\in[T,T+\delta]\times \overline{D}.
\end{cases}
\end{equation}

\section{Maximum principles}
We are now able to derive the sufficient version of the maximum principle.
\subsection{A sufficient maximum principle}
\begin{assumption}\label{+11}
Let $\widehat{u}\in \mathcal{U}^{ad}$ be a control with corresponding solutions $\widehat{X}(t,x)$  to \eqref{SDE} and $(\widehat{p}(t,x),\widehat{q}(t,x),\widehat{r}(t,x,\cdot))$ to \eqref{+5}, respectively. Furthermore, the control and the solutions satisfy
\begin{enumerate}[($\romannumeral1$)]
\item $\widehat{X}\Big|_{t\in[0,T]}\in H_T$;
\item $(\widehat{p},\widehat{q},\widehat{r}(\cdot))\Big|_{(t,x)\in[0,T]\times \overline{D}}\in V\times H\times L^2_{\nu}(H)$;
\item $\mathbb{E}\left[\int_0^T\|\widehat{p}(t,x)\|_V^2+\|\widehat{q}(t,x)\|_H^2ds+\|\widehat{r}(t,x,\cdot)\|_{L^2_{\nu}(H)}^2dt\right]<\infty.$
\end{enumerate}
\end{assumption}

\begin{thm}\label{sufficient}
Suppose that Assumption \ref{+11} holds. For arbitrary $u\in \mathcal{U}$, put
\begin{align*}
H(t,x)=H(t,x,\widehat{X},\widehat{S}_X,u,S_u,\widehat{p},\widehat{q},\widehat{r}(\cdot)),\quad
\widehat{H}(t,x)=H(t,x,\widehat{X},\widehat{S}_X,\widehat{u},\widehat{S}_u,\widehat{p},\widehat{q},\widehat{r}(\cdot)).
\end{align*}
Assume that
\begin{itemize}
\item (Concavity) For each $t\in [0,T]$, the functions
\begin{align*}
(X,S_X,u,S_u)&\rightarrow H(t,x,X,S_X,u,S_u,\widehat{p},\widehat{q},\widehat{r}),\\
X(T)&\rightarrow g(x,X(T))
\end{align*}
are concave a.s..
\item (Maximum condition) For each $t\in [0,T]$,
$$
\mathbb{E}\left[\widehat{H}(t,x)\Big|\mathcal{G}_t\right]=\sup \limits_{u\in \mathcal{U}}\mathbb{E}\left[H(t,x)\Big|\mathcal{G}_t\right],\quad a.s..
$$
\end{itemize}
Then, $\widehat{u}$ is an optimal control.
\end{thm}

\begin{proof}
Consider
\begin{equation}\label{1}
J(u)-J(\widehat{u})=I_1+I_2.
\end{equation}
Here,
\begin{align*}
I_1=&\mathbb{E}\bigg[\int_0^T\int_Df(t,x,X(t,x),\overline{X}(t,x),u(t,x),\overline{u}(t,x))-f(t,x,X(t,x),\overline{X}(t,x),u(t,x),\overline{u}(t,x))dxdt\bigg],\\
I_2=&\int_D\mathbb{E}\left[g(x,X(T,x))-g(x,\widehat{X}(T,x))\right]dx.
\end{align*}
Setting $\widetilde{X}(t,x)=X(t,x)-\widehat{X}(t,x)$ and applying It\^{o}'s formula, one has
\begin{align}\label{2}
I_2\leq&\int_D\mathbb{E}\left[\frac{\partial\widehat{g}}{\partial X}(T,x)\widetilde{X}(T,x)\right]dx=\int_D\mathbb{E}\left[\widehat{p}(T,x)\widetilde{X}(T,x)\right]dx\\ \nonumber
=&\int_D\mathbb{E}\left[\int_0^T\widehat{p}(t,x)d\widetilde{X}(t,x)+\widetilde{X}(t,x)d\widehat{p}(t,x)+\widehat{q}(t,x)\widetilde{\sigma}(t,x)dt
+d\widehat{p}(t,x)d\widetilde{X}(t,x)\right]\\ \nonumber
=&\int_D\mathbb{E}\Bigg\{\int_0^T\widehat{p}(t,x)\left[\widetilde{b}(t,x)+A_x\widetilde{X}(t,x)\right]-\widetilde{X}(t,x)\left[\frac{\partial\widehat{H}}{\partial X}(t,x)+A^{*}_x\widetilde{p}(t,x)+\mathbb{E}\left[\nabla^*_{S}H(t,x)\Big|\mathscr{F}_t\right]\right]\\
&+\widehat{q}(t,x)\widetilde{\sigma}(t,x)dt
+\int_{\mathbb{R}_0}\widehat{r}(t,x,\zeta)\widetilde{\gamma}(t,x,\zeta)\nu(d\zeta,dt)\Bigg\}dx.
\end{align}
By the First Green formula \cite{Wloka1987Partial}, there exist first order boundary differential operators $A_1$ and $A_2$ such that
\begin{equation}\label{3}
\int_D\widehat{p}(t,x)A_x\widetilde{X}(t,x)-\widetilde{X}(t,x)A^{*}_x\widetilde{p}(t,x)dx=\int_{\partial D}\widehat{p}(t,x)A_1\widetilde{X}(t,x)-\widetilde{X}(t,x)A_2\widetilde{p}(t,x)d\mathcal{S}=0.
\end{equation}
Combining \eqref{2}, \eqref{3} and the fact that $u(t)$ is $\mathcal{G}_t$-measurable gives
\begin{align*}
I_2\leq& \int_D\mathbb{E}\Bigg\{\int_0^T\widehat{p}(t,x)\widetilde{b}(t,x)-\widetilde{X}(t,x)\left[\frac{\partial\widehat{H}}{\partial X}(t,x)+\mathbb{E}\left[\nabla^*_{S_X}H(t,x)\Big|\mathscr{F}_t\right]\right]+\widehat{q}(t,x)\widetilde{\sigma}(t,x)dt\\
&+\int_{\mathbb{R}_0}\widehat{r}(t,x,\zeta)\widetilde{\gamma}(t,x,\zeta)\nu(d\zeta,dt)\Bigg\}dx\\
=&-I_1+\int_D\mathbb{E}\left[\int_0^TH(t,x)-\widetilde{X}(t,x)\left[\frac{\partial\widehat{H}}{\partial X}(t,x)+\mathbb{E}\left[\nabla^*_{S_X}H(t,x)\Big|\mathscr{F}_t\right]\right]dt\right]dx\\
\leq& -I_1+\int_D\mathbb{E}\left[\int_0^T\widetilde{u}(t,x)\left[\frac{\partial\widehat{H}}{\partial u}(t,x)+\mathbb{E}\left[\nabla^*_{S_u}H(t,x)\Big|\mathscr{F}_t\right]\right]dt\right]dx\\
=& -I_1+\int_D\mathbb{E}\left[\int_0^T\widetilde{u}(t,x)\frac{\partial\widehat{H}}{\partial u}(t,x)+\widetilde{u}(t,x)\nabla^*_{S_u}H(t,x)dt\right]dx\\
=& -I_1+\int_D\mathbb{E}\left[\int_0^T\widetilde{u}(t,x)\mathbb{E}\left[\frac{\partial\widehat{H}}{\partial u}(t,x)+\nabla^*_{S_u}H(t,x)\Big|\mathcal{G}_t\right]dt\right]dx\\
\leq& -I_1,
\end{align*}
where the last inequality is derived by the maximum condition imposed on $H(t,x)$. This implies that
$$
J(u)-J(\widehat{u})=I_1+I_2\leq 0.
$$
Therefore $\widehat{u}$ becomes the optimal control.
\end{proof}

\subsection{A Necessary Maximum Principle}
We now proceed to study the necessary version of maximum principle.
\begin{assumption}\label{ass nece}
For each $t_0\in[0,T]$ and all bounded $\mathcal{G}_{t_0}$-measurable random variable $\pi(x)$, the process
$\vartheta(t,x)=\pi(x)\mathbb{I}_{(t_0,T]}(t)$ belongs to $\mathcal{U}^{ad}$.
\end{assumption}

\begin{remark}
Thanking to the convex condition imposed on $\mathcal{U}^{ad}$, one has
$$
u^{\epsilon}=\widehat{u}+\epsilon u\in \mathcal{U}^{ad},\;\epsilon\in[0,1]
$$
for any $u,\widehat{u}\in \mathcal{U}^{ad}$.
\end{remark}

Consider the process $Z(t,x)$ obtained by differentiating $X^{\epsilon}(t,x)$ with respect to $\epsilon$ at $\epsilon=0$. Clearly, $Z(t,x)$ satisfies the following equation:
\begin{equation}\label{SPDE nece}
\begin{cases}
dZ(t,x)=&\left[\left(\frac{\partial b}{\partial X}(t,x)+\mathbb{E}\left[\nabla^*_{S_X}b(t,x)\Big|\mathscr{F}_t\right]\right)Z(t,x)
+\left(\frac{\partial b}{\partial u}(t,x)+\mathbb{E}\left[\nabla^*_{S_u}b(t,x)\Big|\mathscr{F}_t\right]\right)u(t,x)\right]dt\\
&+\left[\left(\frac{\partial \sigma}{\partial X}(t,x)+\mathbb{E}\left[\nabla^*_{S_X}\sigma(t,x)\Big|\mathscr{F}_t\right]\right)Z(t,x)
+\left(\frac{\partial \sigma}{\partial u}(t,x)+\mathbb{E}\left[\nabla^*_{S_u}\sigma(t,x)\Big|\mathscr{F}_t\right]\right)u(t,x)\right]dB_t\\
&+\int_{\mathbb{R}_0}\Big[\left(\frac{\partial \gamma}{\partial X}(t,x,\zeta)+\mathbb{E}\left[\nabla^*_{S_X}\gamma(t,x,\zeta)\Big|\mathscr{F}_t\right]\right)Z(t,x)
+\left(\frac{\partial \gamma}{\partial u}(t,x,\zeta)+\mathbb{E}\left[\nabla^*_{S_u}\gamma(t,x,\zeta)\Big|\mathscr{F}_t\right]\right)\\
&u(t,x)\Big]\widetilde{N}(dt,d\zeta)+A_xZ(t,x)dt,\qquad\qquad\qquad\qquad\qquad\qquad\quad(t,x)\in(0,T)\times D,\\
Z(t,x)&=0,\qquad\qquad\qquad\qquad\qquad\qquad\quad\qquad\qquad\qquad\qquad\qquad\qquad(t,x)\in[-\delta,0]\times \overline{D}.
\end{cases}
\end{equation}

\begin{thm}\label{thm nece}
Suppose that Assumptions \ref{+11} and \ref{ass nece} hold. Then£¬ the following equalities are equivalent.
\begin{enumerate}[($\romannumeral1$)]
\item For all bounded $u \in \mathcal{U}^{ad}$,
\begin{equation}\label{nece1}
0=\frac{d}{dt}J(\widehat{u}+\epsilon u)\Big|_{\epsilon=0}.
\end{equation}

\item
\begin{equation}\label{nece2}
0=\int_D\mathbb{E}\left[\frac{\partial H}{\partial u}(t,x)+\nabla^*_{S_u}H(t,x)\Big|\mathcal{G}_t\right]dx\Big|_{u=\widehat{u}},\quad \forall t\in[0,T].
\end{equation}
\end{enumerate}
\end{thm}

\begin{proof}
Assume that \eqref{nece1} holds. Then
\begin{align}\label{5}
0=&\frac{d}{dt}J(\widehat{u}+\epsilon u)\Big|_{\epsilon=0}  \nonumber\\
=&\mathbb{E}\left[\int_0^T\int_D\left[\left(\frac{\partial f}{\partial X}(t,x)+\mathbb{E}\left[\nabla^*_{S_X}f(t,x)\Big|\mathscr{F}_t\right]\right)Z(t,x)
+\left(\frac{\partial f}{\partial u}(t,x)+\mathbb{E}\left[\nabla^*_{S_u}f(t,x)\Big|\mathscr{F}_t\right]\right)u(t,x)\right]dxdt\right]\nonumber\\
&+\mathbb{E}\left[\int_D\frac{\partial \widehat{g}}{\partial X}(T,x)\widehat{Z}(T,x)\rangle dx\right]
\end{align}
where $\widehat{Z}(t,x)$ is the solution to \eqref{SPDE nece}. By It\^{o}'s formula,
\begin{align}\label{6}
&\mathbb{E}\left[\int_D\frac{\partial \widehat{g}}{\partial X}(T,x)\widehat{Z}(T,x)\rangle dx\right]=\mathbb{E}\left[\int_D\widehat{p}(T,x)\widehat{Z}(T,x)dx\right]\nonumber\\
=&\mathbb{E}\bigg[\int_Ddx\int_0^T\widehat{p}(t,x)d\widehat{Z}(t,x)+\widehat{Z}(t,x)d\widehat{p}(t,x)+d\widehat{Z}(t,x)d\widehat{p}(t,x)\bigg]\nonumber\\
=&\mathbb{E}\bigg[\int_Ddx\int_0^T\widehat{p}(t,x)\Big[\left(\frac{\partial b}{\partial X}(t,x)+\mathbb{E}\left[\nabla^*_{S_X}b(t,x)\Big|\mathscr{F}_t\right]\right)Z(t,x)
+\left(\frac{\partial b}{\partial u}(t,x)+\mathbb{E}\left[\nabla^*_{S_u}b(t,x)\Big|\mathscr{F}_t\right]\right)u(t,x)\nonumber\\
&+A_xZ(t,x)\Big]dt-\Big(\frac{\partial H}{\partial X}(t,x)+A^{*}_xp(t,x)+\mathbb{E}\left[\nabla^*_{S_X}H(t,x)\Big|\mathscr{F}_t\right]\Big)\widehat{Z}(t,x)dt+\widehat{q}(t,x)\Big[\Big(\frac{\partial \sigma}{\partial X}(t,x) \nonumber\\
&+\mathbb{E}\left[\nabla^*_{S_X}\sigma(t,x)\Big|\mathscr{F}_t\right]\Big)Z(t,x)+\left(\frac{\partial \sigma}{\partial u}(t,x)+\mathbb{E}\left[\nabla^*_{S_u}\sigma(t,x)\Big|\mathscr{F}_t\right]\right)u(t,x)\Big]d t+\int_{\mathbb{R}_0}\widehat{r}(t,x,\zeta)\Big[\Big(\frac{\partial \gamma}{\partial X}(t,x,\zeta)
\nonumber\\
&+\mathbb{E}\left[\nabla^*_{S_X}\gamma(t,x,\zeta)\Big|\mathscr{F}_t\right]\Big)Z(t,x)
+\Big(\frac{\partial \gamma}{\partial u}(t,x,\zeta)+\mathbb{E}\left[\nabla^*_{S_u}\gamma(t,x,\zeta)\Big|\mathscr{F}_t\right]\Big)u(t,x)\Big]\nu(dt,d\zeta)\bigg]\nonumber\\
=&-\mathbb{E}\left[\int_0^T\int_D\left[\left(\frac{\partial f}{\partial X}(t,x)+\mathbb{E}\left[\nabla^*_{S_X}f(t,x)\Big|\mathscr{F}_t\right]\right)Z(t,x)
+\left(\frac{\partial f}{\partial u}(t,x)+\mathbb{E}\left[\nabla^*_{S_u}f(t,x)\Big|\mathscr{F}_t\right]\right)u(t,x)\right]dxdt\right]\nonumber\\
&+\mathbb{E}\left[\int_0^T\int_D\frac{\partial H}{\partial u}(t,x)u(t,x)+\mathbb{E}\left[\nabla^*_{S_u}H(t,x)\Big|\mathscr{F}_t\right]u(t,x)dxdt\right],
\end{align}
where the last step follows from the first Green formula \cite{Wloka1987Partial}.

Combining \eqref{5} and \eqref{6}, one has
$$
0=\mathbb{E}\left[\int_0^T\int_D\frac{\partial H}{\partial u}(t,x)u(t,x)+\mathbb{E}\left[\nabla^*_{S_u}H(t,x)\Big|\mathscr{F}_t\right]u(t,x)dxdt\right].
$$
Now we set $u(t,x)=\pi(x)\mathbb{I}_{(t_0,T]}(t)$, where $\pi(x)$ is a bounded $\mathcal{G}_{t_0}$-measurable random variable. Then, we have
\begin{align*}
0= &\int_0^T\mathbb{E}\left[\int_D\frac{\partial H}{\partial u}(t,x)\pi(x)\mathbb{I}_{(t_0,T]}(t)+\mathbb{E}\left[\nabla^*_{S_u}H(t,x)\Big|\mathscr{F}_t\right]\pi(x)\mathbb{I}_{(t_0,T]}(t)\right]dt\\
=&\int_{t_0}^T\mathbb{E}\left[\int_D\frac{\partial H}{\partial u}(t,x)\pi(x)+\mathbb{E}\left[\nabla^*_{S_u}H(t,x)\Big|\mathscr{F}_t\right]\pi(x)\right]dt\\
=&\int_{t_0}^T\mathbb{E}\left[\int_D\frac{\partial H}{\partial u}(t,x)\pi(x)+\nabla^*_{S_u}H(t,x)\pi(x)\right]dt.
\end{align*}
Differentiating with respect to $t_0$, it follows that
$$
0=\mathbb{E}\left[\int_D\frac{\partial H}{\partial u}(t_0,x)\pi(x)+\nabla^*_{S_u}H(t_0,x)\pi(x)dx\right],\quad \forall t_0\in[0,T].
$$
Since this holds for all such $\pi(x)$, we have
$$
0=\int_D\mathbb{E}\left[\frac{\partial H}{\partial u}(t_0,x)+\nabla^*_{S_u}H(t_0,x)\Big|\mathcal{G}_t\right]dx,\quad \forall t_0\in[0,T].
$$
The argument above is reversible. Thus \eqref{nece1} and \eqref{nece2} are equivalent.
\end{proof}

\section{Existence and Uniqueness}

In this section, we prove the existence and uniqueness of the solution to the following general BSPDE \eqref{+5} with spatial-temporal dependence:
\begin{equation}\label{BSPDE}
\begin{cases}
dp(t,x)&=-\left(A_xp(t,x)-\mathbb{E}[F(t)|\mathscr{F}_t]\right)dt+q(t,x)dB_t+\int_{\mathbb{R}_0}r(t,x,\zeta)\widetilde{N}(dt,d\zeta),\\
&\qquad\qquad\qquad\qquad\qquad\qquad\;\:\qquad\quad\qquad(t,x)\in[0,T]\times D;\\
p(t,x)&=\theta(t,x),\qquad\quad\qquad\qquad\qquad\quad\:\:\qquad\quad (t,x)\in[T,T+\delta]\times \overline{D};\\
p(t,x)&=\chi(t,x), \qquad\qquad\qquad\qquad\qquad\:\;\quad\qquad (t,x)\in[0,T)\times \partial D;\\
q(t,x)&=0, \qquad\qquad\qquad\qquad\qquad\quad\qquad\;\;\,\qquad (t,x)\in[T,T+\delta]\times \overline{D};\\
r(t,x,\cdot)&=0, \qquad\qquad\qquad\qquad\qquad\quad\qquad\;\;\,\qquad (t,x)\in[T,T+\delta]\times \overline{D}.
\end{cases}
\end{equation}
Here $F=F(t):[0,T+\delta]\times \mathbb{R}\times \mathbb{R}\times \mathbb{R}\times \mathbb{R}\times \mathbb{R}\times \mathbb{R}\rightarrow \mathbb{R}$ is a functional on $C^1(H)$ as follows:
$$
F(t)=F(t,p(t,x),\overline{p}(t+\delta,x),q(t,x),\overline{q}(t+\delta,x),r(t,x,\cdot),\overline{r}(t+\delta,x,\cdot)).
$$

\begin{assumption}\label{Ax condition}
Assume that $A_x:V \rightarrow V^{*}$ is a bounded and linear operator. Moreover, there exists two constants $\alpha_1>0$ and $\alpha_2 \geq 0$ such that
\begin{equation}\label{7}
2\langle A_xu,u\rangle_{*}+\alpha_1 \|u\|^2_V\leq \alpha_2||u||^2_H, \quad \forall u\in V.
\end{equation}
\end{assumption}

\begin{assumption}\label{exi and uni}
Suppose that the following assumptions hold:
\begin{enumerate}[($\romannumeral1$)]
\item $\theta(t,x)$ is a given $\mathscr{F}_t$-measurable process such that
$$
\mathbb{E}\left[\sup \limits_{t\in[T,T+\delta]}\|\theta(t,x)\|_H^2\right]< \infty;
$$
\item $F(t,0,0,0,0,0,0)\in H_{T}$;
\item  For any $t,p_1,q_1,r_1,p_2,q_2,r_2$, there is a constant $C>0$ such that
\begin{align*}
&|F(t,p_1,\overline{p}_1,q_1,\overline{q}^{\delta}_1,r_1,\overline{r}^{\delta}_1)
    -F(t,p_2,\overline{p}^{\delta}_2,q_2,\overline{q}^{\delta}_2,r_2,\overline{r}^{\delta}_2)|^2\\
\leq&C\left(|p_1-p_2|^2+|q_1-q_2|^2+\int_{\mathbb{R}_0}|r_1-r_2|^2\nu(d\zeta)+|\overline{p}^{\delta}_1-\overline{p}^{\delta}_2|^2+|\overline{q}^{\delta}_1-\overline{q}^{\delta}_2|^2
+\int_{\mathbb{R}_0}|\overline{r}^{\delta}_1-\overline{r}^{\delta}_2|^2\nu(d\zeta)\right).
\end{align*}
\end{enumerate}
\end{assumption}

In the sequel, we use $C$ to represent the constant large enough such that all the inequalities are satisfied.

\begin{thm}\label{exi and uni thm}
Under Assumptions \ref{Ax condition} and \ref{exi and uni}, BSPDEs \eqref{BSPDE} has a unique solution $(p,q,r(\cdot))$ such that the restriction on $(t,x)\in [0,T]\times \overline{D}$ satisfies
\begin{enumerate}[($\romannumeral1$)]
\item $(p,q,r(\cdot))\Big|_{(t,x)\in[0,T]\times \overline{D}}\in V\times H\times L^2_{\nu}(H)$,
\item $\mathbb{E}\left[\int_0^T\|p(t,x)\|_V^2+\|q(t,x)\|_H^2ds+\|r(t,x,\cdot)\|_{L^2_{\nu}(H)}^2dt\right]<\infty.$
\end{enumerate}
\end{thm}

\begin{proof}
We decompose the proof into five steps.

{\bf Step 1:} Assume that the driver $F(t)$ is independent of $p$ and $\overline{p}$ such that $(p,q,r(\cdot))\in V\times H\times L^2_{\nu}(H)$ satisfies
\begin{equation}\label{+13}
\begin{cases}
dp(t,x)=&-\left(A_xp(t,x)-\mathbb{E}[F(t,q(t,x),\overline{q}(t+\delta,x),r(t,x,\cdot),\overline{r}(t+\delta,x,\cdot)|\mathscr{F}_t]\right)dt+q(t,x)dB_t\\
&+\int_{\mathbb{R}_0}r(t,x,\zeta)\widetilde{N}(dt,d\zeta),\qquad\qquad\qquad\quad\,\,\qquad(t,x)\in[0,T]\times D;\\
p(t,x)=&\zeta(t,x),\qquad\quad\qquad\qquad\qquad\quad\:\:\qquad\quad\qquad\quad\:\,(t,x)\in[T,T+\delta]\times \overline{D};\\
p(t,x)=&\theta(t,x), \qquad\qquad\qquad\qquad\qquad\:\;\quad\qquad\qquad\quad\:\, (t,x)\in[0,T)\times \partial D;\\
q(t,x)=&0, \qquad\qquad\qquad\qquad\qquad\quad\qquad\;\;\,\qquad\qquad\quad\:\, (t,x)\in[T,T+\delta]\times \overline{D};\\
r(t,x,\cdot)=&0, \qquad\qquad\qquad\qquad\qquad\quad\qquad\;\;\,\qquad\qquad\quad\:\, (t,x)\in[T,T+\delta]\times \overline{D}.
\end{cases}
\end{equation}
We first prove the uniqueness and  existence of solutions to \eqref{+13}. By Theorem 4.2 in \cite{oksendal2012Optimal}, it is easy to show that for each fixed $n\in \mathbb{N}$, there exists a unique solution to the following stochastic partial differential equation
\begin{equation}\label{BSPDE n}
\begin{cases}
dp^{n+1}(t,x)=&\mathbb{E}\Big[F\left(t,q^n(t,x),\overline{q}^n(t+\delta,x),r^n(t,x,\cdot),\overline{r}^n(t+\delta,x,\cdot)\right)\Big|\mathscr{F}_t\Big]dt-A_xp^{n+1}(t,x)dt\\
&+q^{n+1}(t,x)dB_t+\int_{\mathbb{R}_0}r^{n+1}(t,x,\zeta)\widetilde{N}(dt,d\zeta),\qquad\qquad(t,x)\in[0,T]\times D;\\
p^{n+1}(t,x)=&\zeta(t,x),\qquad\quad\qquad\qquad\quad\:\:\qquad\qquad\qquad\qquad\qquad\quad\;\, (t,x)\in[T,T+\delta]\times \overline{D};\\
p^{n+1}(t,x)=&\theta(t,x), \qquad\qquad\qquad\qquad\:\;\qquad\qquad\qquad\qquad\qquad\quad\;\, (t,x)\in[0,T)\times \partial D;\\
q^{n+1}(t,x)=&0, \qquad\qquad\qquad\qquad\qquad\;\;\,\qquad\qquad\qquad\qquad\qquad\quad\;\, (t,x)\in[T,T+\delta]\times \overline{D};\\
r^{n+1}(t,x,\cdot)=&0, \qquad\qquad\qquad\qquad\qquad\;\;\,\qquad\qquad\qquad\qquad\qquad\quad\;\, (t,x)\in[T,T+\delta]\times \overline{D}.
\end{cases}
\end{equation}
such that
$$
(p^n,q^n,r^n(\cdot))\Big|_{(t,x)\in[0,T]\times \overline{D}}\in V\times H\times L^2_{\nu}(H).
$$
Here, $q^{0}(t,x)=r^{0}(t,x,\cdot)=0$ for all $(t,x)\in [0,T+\delta]\times \bar{D}$.

We now aim to show that $(p^n,q^n,r^n(\cdot))$ forms a Cauchy sequence. By similar arguments in Proposition \ref{+6}, we have
\begin{align*}
&\mathbb{E}\left[\int_t^T\|\overline{p}^{n+1}(s+\delta,x)-\overline{p}^n(s+\delta,x)\|^2_{H}ds\right]\\
=&\mathbb{E}\left[\int_D\int_t^T\int_{R_{\theta}}\int_{s}^{s+\delta}|Q(s+\delta,\varsigma,x,y)|^2|p^{n+1}(\varsigma,x+y)-p^{n}(\varsigma,x+y)|^2 d\zeta dydxds\right]\\
\leq&\mathbb{E}\left[\int_D\int_{t}^{T+\delta}\left(\int_{D\cap(z-R_{\theta})}\int_{(\varsigma-\delta)\vee t}^{\varsigma\wedge T}|Q(s+\delta,\varsigma,x,z-x)|^2 dsdx\right)|p^{n+1}(\varsigma,z)-p^{n}(\varsigma,z)|^2d\varsigma dz\right]\\
\leq& C\mathbb{E}\left[\int_D\int_{t}^{T+\delta}|p^{n+1}(\varsigma,z)-p^{n}(\varsigma,z)|^2d\varsigma dz\right]=C\mathbb{E}\left[\int_{t}^{T}\|p^{n+1}(s,x)-p^n(s,x)\|_{H}^2ds\right].
\end{align*}
Similarly, one has
\begin{equation}\label{91}
\mathbb{E}\left[\int_t^T\|\overline{q}^{n+1}(s+\delta,x)-\overline{q}^n(s+\delta,x)\|^2_{H}ds\right]\leq C\mathbb{E}\left[\int_{t}^{T}\|q^{n+1}(s,x)-q^n(s,x)\|_{H}^2ds\right].
\end{equation}
and
\begin{equation}\label{92}
\mathbb{E}\left[\int_t^T\|\overline{\gamma}^{n+1}(s+\delta,x,\cdot)-\overline{\gamma}^n(s+\delta,x,\cdot)\|^2_{H}ds\right]\leq C\mathbb{E}\left[\int_{t}^{T}\|\gamma^{n+1}(s,x,\cdot)-\gamma^n(s,x,\cdot)\|_{L^2_{\nu}(H)}^2ds\right].
\end{equation}

For simplicity, we can write
\begin{align*}
F_n(t)&=F(t,q^{n}(t,x),\overline{q}^{n}(t+\delta,x),r^{n}(t,x,\cdot),\overline{r}^{n}(t+\delta,x,\cdot));\\
L^n(s)&=\|q^{n}(s,x)-q^{n-1}(s,x)\|_H^2+\|r^{n}(s,x,\cdot)-r^{n-1}(s,x,\cdot)\|_{L^2_{\nu}(H)}^2.
\end{align*}
{\bf Step 2:} Applying It\^{o}'s formula to $\|p^{n+1}(t,x)-p^n(t,x)\|_H^2$, we have
\begin{align}\label{pr1}
&-\|p^{n+1}(t,x)-p^n(t,x)\|_H^2 \nonumber\\
=&2\int_t^T\langle p^{n+1}(s,x)-p^n(s,x),A_x(p^{n+1}(s,x)-p^n(s,x))\rangle_{*} ds \nonumber\\
&-2\int_t^T\langle p^{n+1}(s,x)-p^n(s,x),\mathbb{E}\left[F^{n}(s)-F^{n-1}(s)\Big|\mathscr{F}_s\right]\rangle_H ds+\int_t^T\|q^{n+1}(s,x)-q^n(s,x)\|_{L^2_{\nu}(H)}^2ds \nonumber\\
&+2\int_t^T\langle p^{n+1}(s,x)-p^n(s,x),q^{n+1}(s,x)-q^n(s,x)\rangle_H dB_s+\int_t^T\|r^{n+1}(s,x,\cdot)-r^n(s,x,\cdot)\|_{L_{\nu}(H)}^2ds \nonumber\\
&+2\int_t^T\int_{\mathbb{R}_0}\langle p^{n+1}(s,x)-p^n(s,x),r^{n+1}(s,x,\zeta)-r^n(s,x,\zeta) \rangle_H \widetilde{N}(ds,d\zeta).
\end{align}
By the Lipschitz condition imposed on $F$, \eqref{91} and \eqref{92}, for any fixed constant $\rho> 0$, one has
\begin{align}\label{pr2}
&-2\mathbb{E}\Big[\int_t^T\langle p^{n+1}(s,x)-p^n(s,x),\mathbb{E}\left[F^{n}(s)-F^{n-1}(s)\Big|\mathscr{F}_s\right]\rangle_H ds\Big] \nonumber\\
=&-2\mathbb{E}\Big[\int_t^T\langle p^{n+1}(s,x)-p^n(s,x),F^{n}(s)-F^{n-1}(s)\rangle_H ds\Big] \nonumber\\
\leq& 2\mathbb{E}\Big[\int_t^T\|p^{n+1}(s,x)-p^n(s,x)\|_H\|F^{n}(s)-F^{n-1}(s)\|_Hds\Big] \nonumber\\
\leq& \frac{1}{\rho}\mathbb{E}\Big[\int_t^T\|p^{n+1}(s,x)-p^n(s,x)\|_H^2\Big]+\rho \mathbb{E}\Big[\int_t^T\|F^{n}(s)-F^{n-1}(s)\|_H^2ds\Big] \nonumber\\
\leq& \frac{1}{\rho}\mathbb{E}\Big[\int_t^T\|p^{n+1}(s,x)-p^n(s,x)\|_H^2ds\Big]+ \rho C\mathbb{E}\Big[\int_t^T\|q^{n}(s,x)-q^{n-1}(s,x)\|_H^2\nonumber\\
&+\|r^{n}(s,x,\cdot)-r^{n-1}(s,x,\cdot)\|_{L_{\nu}(H)}^2+\|\overline{q}^{n+1}(s+\delta,x)-\overline{q}^n(s+\delta,x)\|^2_{H}\nonumber\\
&+\|\overline{\gamma}^{n+1}(s+\delta,x,\cdot)-\overline{\gamma}^n(s+\delta,x,\cdot)\|^2_{H}ds\Big]\nonumber\\
\leq&\frac{1}{\rho}\mathbb{E}\Big[\int_t^T\|p^{n+1}(s,x)-p^n(s,x)\|_H^2ds\Big]+\rho C\mathbb{E}\left[\int_t^TL^n(s)ds\right].
\end{align}
Taking expectation on both sides of \eqref{pr1}, and applying both \eqref{7} and \eqref{pr2}, we have
\begin{align}\label{pr8}
&\mathbb{E}\|p^{n+1}(t,x)-p^n(t,x)\|_H^2 \nonumber\\
\leq& \mathbb{E}\left[\int_t^T\alpha_2\|p^{n+1}(s,x)-p^n(s,x)\|_H^2-\alpha_1\|p^{n+1}(s,x)-p^n(s,x)\|_V^2ds\right]\nonumber\\
&+\frac{1}{\rho}\mathbb{E}\left[\int_t^T\|p^{n+1}(s,x)-p^n(s,x)\|_H^2ds\right]-\mathbb{E}\left[\int_t^T\|q^{n+1}(s,x)-q^n(s,x)\|_H^2ds\right]\nonumber\\
&+\rho C\mathbb{E}\left[\int_t^TL^n(s)ds\right]-\mathbb{E}\left[\int_t^T\|r^{n+1}(s,x,\cdot)-r^n(s,x,\cdot)\|_{L_{\nu}(H)}^2ds\right]\nonumber\\
\leq& (\frac{1}{\rho}+\alpha_2)\mathbb{E}\left[\int_t^T\|p^{n+1}(s,x)-p^n(s,x)\|_H^2ds\right]+\rho C\mathbb{E}\left[\int_t^TL^n(s)ds\right]\nonumber\\
&-\mathbb{E}\left[\int_t^TL^{n+1}(s)ds\right]-\alpha_1\mathbb{E}\left[\int_t^T\|p^{n+1}(s,x)-p^n(s,x)\|_V^2ds\right].
\end{align}
We can choose a $\rho>0$ such that
\begin{align}\label{pr4}
&\mathbb{E}\|p^{n+1}(t,x)-p^n(t,x)\|_H^2+\mathbb{E}\left[\int_t^TL^{n+1}(s)ds\right]+\alpha_1\mathbb{E}\left[\int_t^T\|p^{n+1}(s,x)-p^n(s,x)\|_V^2ds\right]\nonumber\\ \leq&\frac{1}{2}\mathbb{E}\left[\int_t^TL^n(s)ds\right]+\alpha_3\mathbb{E}[\int_t^T\|p^{n+1}(s,x)-p^n(s,x)\|_H^2ds].
\end{align}
where $\alpha_3=\alpha_2+\frac{1}{\rho}$.
Multiplying by $e^{\alpha_3 t}$ and integrating both sides in $[0,T]$, we have
\begin{align*}
&\int_0^Te^{\alpha_3 t}\mathbb{E}\left[\|p^{n+1}(t,x)-p^n(t,x)\|_H^2-\alpha_3\int_t^T\|p^{n+1}(s,x)-p^n(s,x)\|_H^2ds\right]dt\nonumber\\
&+\int_0^Te^{\alpha_3 t}\mathbb{E}\left[\int_t^T L^{n+1}(s)ds\right]dt\leq\frac{1}{2}\int_0^Te^{\alpha_3 t}\mathbb{E}\left[\int_t^T L^n(s)ds\right]dt
\end{align*}
and so
\begin{equation}\label{10}
\|p^{n+1}(t,x)-p^n(t,x)\|_{H_T}^2+\int_0^Te^{\alpha_3 t}\mathbb{E}\left[\int_t^T L^{n+1}(s)ds\right]dt\leq \frac{1}{2}\int_0^Te^{\alpha_3 t}\mathbb{E}\left[\int_t^T L^n(s)ds\right]dt.
\end{equation}
In particular,
\begin{equation}\label{9}
\int_0^Te^{\alpha_3 t}\mathbb{E}\left[\int_t^T L^{n+1}(s)ds\right]dt\leq \frac{1}{2}\int_0^Te^{\alpha_3 t}\mathbb{E}\left[\int_t^T L^n(s)ds\right]dt\leq C(\frac{1}{2})^n.
\end{equation}

{\bf Step 3: Existence}

Substituting \eqref{9} into \eqref{10}, one has
$$
\|p^{n+1}(t,x)-p^n(t,x)\|_{H_T}^2\leq C(\frac{1}{2})^n.
$$
It follows from \eqref{pr4} that
\begin{align*}
\mathbb{E}\left[\int_t^T L^{n+1}(s)ds\right]&\leq C(\frac{1}{2})^n+\frac{1}{2}\mathbb{E}\left[\int_t^T L^{n}(s)ds\right]\\
&\leq C(\frac{1}{2})^n+C(\frac{1}{2})^n+\frac{1}{2^2}\mathbb{E}\left[\int_t^T L^{n-1}(s)ds\right]\\
&\leq \frac{nC}{2^n}+\frac{1}{2^n}\mathbb{E}\left[\int_t^TL^{1}(s)ds\right]\\
&\leq \frac{nC}{2^n}+\frac{C}{2^n}=\frac{n(C+1)}{2^n}.
\end{align*}
Letting $n\rightarrow\infty$, we have
\begin{equation}\label{+7}
0=\lim \limits_{n\rightarrow\infty}\mathbb{E}\left[\int_t^T L^{n+1}(s)ds\right].
\end{equation}
Substituting \eqref{+7} into \eqref{pr4} yields
$$
\lim \limits_{n\rightarrow\infty}\mathbb{E}\left[\int_t^T\|p^{n+1}(s,x)-p^n(s,x)\|_V^2ds\right]=0.
$$
Therefore, for $(t,x)\in[0,T]\times\overline{D}$, the sequence $\{(p^n,q^n,r^n(\cdot))\in V\times H\times L^2_{\nu}(H)\}$ converges to $(p,q,r(\cdot))\in V\times H\times L^2_{\nu}(H)$. Letting $n\rightarrow\infty$ in \eqref{BSPDE n}, we can see that the limit $(p,q,r(\cdot))=\lim \limits_{n\rightarrow\infty}(p^n,q^n,r^n(\cdot))$ is indeed the solution to \eqref{+13} on $(t,x)\in[0,T]\times\overline{D}$.

{\bf Step 4: Uniqueness}

Suppose that $(p,q,r(\cdot))$, $(p^{(0)},q^{(0)},r^{(0)}(\cdot))$ are two solutions to \eqref{+13}. As the same arguments in {\bf Step 2},
we see that
\begin{align*}
&\mathbb{E}\|p(t,x)-p^{(0)}(t,x)\|_H^2+\alpha_1\mathbb{E}\left[\int_t^T\alpha_1\|p(s,x)-p^{(0)}(s,x)\|_V^2ds\right]+\frac{1}{2}\mathbb{E}\bigg[\int_t^T\|p(t,x)-p^{(0)}(t,x)\|_H^2\\
&+\|q(s,x)-q^{(0)}(s,x)\|_H^2+\|r(s,x,\cdot)-r^{(0)}(s,x,\cdot)\|_H^2ds\bigg]\leq\alpha_3\mathbb{E}\left[\int_t^T\|p(s,x)-p^{(0)}(s,x)\|_H^2ds\right].
\end{align*}
It follows that
$$
\mathbb{E}\|p(t,x)-p^{(0)}(t,x)\|_H^2\leq  \alpha_3\mathbb{E}\left[\int_t^T\|p(s,x)-p^{(0)}(s,x)\|_H^2ds\right].
$$
By Gronwall's Lemma, we know that $\mathbb{E}\|p(t,x)-p^{(0)}(t,x)\|_H^2=0$ and $p(t,x)=p^{(0)}(t,x)$ a.s. and so
\begin{align*}
&\mathbb{E}\left[\int_t^T\alpha_1\|p(s,x)-p^{(0)}(s,x)\|_V^2ds\right]+\frac{1}{2}\mathbb{E}\left[\int_t^T\|q(s,x)-q^{(0)}(s,x)\|_H^2ds\right] \\
&+\frac{1}{2}\mathbb{E}\left[\int_t^T\|r(s,x,\cdot)-r^{(0)}(s,x,\cdot)\|_{L_{\nu}(H)}^2ds\right]\leq 0,
\end{align*}
which implies $q(t,x)=q^{(0)}(t,x)$ and $r(t,x,\cdot)=r^{(0)}(t,x,\cdot)$ a.s..

{\bf Step 5: General case}

Consider the following iteration with general driver $F$:
\begin{equation}\label{BSPDE n2}
\begin{cases}
dp^{n+1}(t,x)&=-A_xp^{n+1}(t,x)dt+\mathbb{E}\Big[F\left(t,p^n(t,x),\overline{p}^n(t+\delta,x),q^{n+1}(t,x),\overline{q}^{n+1}(t+\delta,x),\right.\\
&\quad\left.r^{n+1}(t,x,\cdot),\overline{r}^{n+1}(t+\delta,x,\cdot)\right)\Big|\mathscr{F}_t\Big]dt+q^{n+1}(t,x)dB_t+\int_{\mathbb{R}_0}r^{n+1}(t,x,\zeta)\widetilde{N}(dt,d\zeta),\\
&\qquad\qquad\qquad\qquad\qquad\qquad\qquad\qquad\qquad\qquad\qquad\;\;(t,x)\in[0,T]\times D;\\
p^{n+1}(t,x)&=\zeta(t,x),\qquad\quad\qquad\qquad\quad\:\:\qquad\qquad\qquad\qquad\qquad (t,x)\in[T,T+\delta]\times \overline{D};\\
p^{n+1}(t,x)&=\theta(t,x), \qquad\qquad\qquad\qquad\:\;\qquad\qquad\qquad\qquad\qquad (t,x)\in[0,T)\times \partial D;\\
q^{n+1}(t,x)&=0, \qquad\qquad\qquad\qquad\qquad\;\;\,\qquad\qquad\qquad\qquad\qquad (t,x)\in[T,T+\delta]\times \overline{D};\\
r^{n+1}(t,x,\cdot)&=0, \qquad\qquad\qquad\qquad\qquad\;\;\,\qquad\qquad\qquad\qquad\qquad (t,x)\in[T,T+\delta]\times \overline{D},
\end{cases}
\end{equation}
where $p^0(t,x)=0$. Similar to the proofs of {\bf Steps 1-2}, we can easily obtain the following inequality:
\begin{align*}
&\mathbb{E}\|p^{n+1}(t,x)-p^n(t,x)\|_H^2+(1-\rho C)\mathbb{E}\left[\int_t^TL^{n+1}(s)ds\right]+\alpha_1\mathbb{E}\left[\int_t^T\|p^{n+1}(s,x)-p^n(s,x)\|_V^2ds\right]\\
&-\alpha_3\mathbb{E}[\int_t^T\|p^{n+1}(s,x)-p^n(s,x)\|_H^2ds]\leq \rho C\mathbb{E}\left[\int_t^T\|p^{n}(s,x)-p^{n-1}(s,x)\|_H^2ds\right].
\end{align*}
Choosing $\rho=\frac{1}{2C}$, we have
\begin{align*}
&\mathbb{E}\|p^{n+1}(t,x)-p^n(t,x)\|_H^2-\alpha_3\mathbb{E}[\int_t^T\|p^{n+1}(s,x)-p^n(s,x)\|_H^2ds]\\
\leq&\mathbb{E}\|p^{n+1}(t,x)-p^n(t,x)\|_H^2+\frac{1}{2}\mathbb{E}\left[\int_t^TL^{n+1}(s)ds\right]+\alpha_1\mathbb{E}\left[\int_t^T\|p^{n+1}(s,x)-p^n(s,x)\|_V^2ds\right]\\
&-\alpha_3\mathbb{E}[\int_t^T\|p^{n+1}(s,x)-p^n(s,x)\|_H^2ds]\\
\leq& \frac{1}{2}\mathbb{E}\left[\int_t^T\|p^{n}(s,x)-p^{n-1}(s,x)\|_H^2ds\right].
\end{align*}
Multiplying by $e^{\alpha_3 t}$ and integrating both sides in $[0,T]$, one has
\begin{align*}
\mathbb{E}\left[\int_{\tau}^T\|p^{n+1}(s,x)-p^n(s,x)\|_H^2ds\right]
\leq& \frac{1}{2}\int_{\tau}^Te^{\alpha_3 t}\mathbb{E}\left[\int_t^T\|p^{n}(s,x)-p^{n-1}(s,x)\|_H^2dsdt\right]\\
\leq& C\mathbb{E}\left[\int_{\tau}^T\int_t^T\|p^{n}(s,x)-p^{n-1}(s,x)\|_H^2dsdt\right]\\
\leq& C\int_{\tau}^T\mathbb{E}\left[\int_{t}^T\|p^{n}(s,x)-p^{n-1}(s,x)\|_H^2ds\right]dt.
\end{align*}
Note that for any $\tau\in[0,T]$,
$$
\mathbb{E}\left[\int_{\tau}^T\|p^{2}(s,x)-p^1(s,x)\|_H^2ds\right]\leq C\int_{\tau}^T\mathbb{E}\left[\int_{0}^T\|p^{1}(s,x)-p^0(s,x)\|_H^2ds\right]\leq C^2(T-\tau).
$$
Iterating the above inequality shows that
\begin{equation}\label{+15}
\mathbb{E}\left[\int_{\tau}^T\|p^{n+1}(s,x)-p^n(s,x)\|_H^2ds\right]\leq \frac{C^{n+1}(T-\tau)^n}{n!}.
\end{equation}
Using \eqref{+15} and a similar argument as in {\bf Steps 3-4}, it can see that the limit $(p,q,r(\cdot))=\lim \limits_{n\rightarrow\infty}(p^n,q^n,r^n(\cdot))$ is indeed the solution to \eqref{+5} on $(t,x)\in[0,T]\times\overline{D}$. This ends the proof.
\end{proof}

\section{Applications}
In this section, as applications, our results obtained in the previous sections are applied to the stochastic population dynamic models with spatial-temporal dependence.

As discussed in \cite{Oksendal2005optimal}, the following SPDE is a natural model for population growth:
\begin{equation}\label{exam 0}
\begin{cases}
dX(t,x)&=\left[\frac{1}{2}\Delta X(t,x)+\alpha{X}(t,x)-u(t,x)\right]dt+\beta{X}(t,x)dB_t,\,(t,x)\in[0,T]\times D;\\
X(0,x)&=\xi(t,x), \qquad\qquad\qquad\qquad\qquad\quad\quad \qquad\qquad\qquad\;\,\;\,\;\; x\in \overline{D};\\
X(t,x)&=\eta(t,x)\geq0, \qquad\qquad\qquad\qquad\qquad\:\!\qquad\qquad\qquad\;\:\;\,\;\; (t,x)\in(0,T]\times \partial D.\\
\end{cases}
\end{equation}
Here, $X(t,x)$ is the density of a population (e.g. fish) at $(t,x)$, $u(t,x)$ is the harvesting rate at $(t,x)\in [0,T]\times \overline{D}$, and
$$
\Delta=\sum \limits_{i=1}^n \frac{\partial^2}{\partial{x_i}^2}
$$
is a Laplacian operator. This type equation is called the stochastic reaction-diffusion equation.

We improve \eqref{exam 0} to the following system:
\begin{equation}\label{exam max}
\begin{cases}
dX(t,x)&=\left[\frac{1}{2}\Delta X(t,x)-u(t,x)\right]dt+\left(\gamma_1(t,x)X(t,x)+\gamma_2(t,x)\overline{X}(t,x)\right)\left(\gamma_3(t,x)dt\right.\\
&\quad\left.+\gamma_4(t,x)dB_t+\int_{\mathbb{R}_0}\gamma_5(t,x,\zeta)\widetilde{N}(dt,d\zeta)\right],\quad\qquad\quad\:(t,x)\in[0,T]\times D;\\
X(t,x)&=\xi(t,x), \qquad\qquad\qquad\qquad\qquad\quad\quad \qquad\qquad\qquad\;\,\;(t,x)\in[-\delta,0]\times \overline{D};\\
X(t,x)&=\eta(t,x)\geq0, \qquad\qquad\qquad\qquad\qquad\:\!\qquad\qquad\qquad\;\:\;(t,x)\in(0,T]\times \partial D;\\
u(t,x)&=\beta(t,x)\geq0, \qquad\qquad\qquad\qquad\qquad\:\!\qquad\qquad\qquad\;\:\,(t,x)\in[-\delta,0]\times \overline{D},
\end{cases}
\end{equation}
where $\gamma_i(t,x)\in H_T\; (i=1,2,3,4)$, $\int_{\mathbb{R}_0}\gamma_5(t,x,\zeta)\widetilde{N}(dt,d\zeta)\in H_T$ are all given.

Now, we give two examples for the stochastic optimal control problems governed by \eqref{exam max} under different performance functionals.
\begin{example}\label{+9}
We consider the following performance functional:
\begin{equation}\label{per func}
J_0(u)=\mathbb{E}\left[\int_D\int_0^T\log \left(u(t,x)\right)dtdx+\int_Dk(x)\log \left(X(T,x)\right)dx\right].
\end{equation}
Here, $\beta \in(0,1)$ is a constant and $k(x)\in H$ is a nonnegative, $\mathscr{F}_T$-measurable process. Now, we aim to find $\widehat{u}(t,x)\in \mathcal{U}^{ad}$ such that
$$
J_0(\widehat{u})=\sup \limits_{u\in\mathcal{U}^{ad}}J_0(u).
$$
Since the Laplacian operator $\Delta$ is self-adjoint,  the Hamiltonian functional associated to this problem takes the following form
\begin{align*}
H(t,x,S,z,u,p,q,r(\cdot))=&\log u(t,x)+[\gamma_1(t,x)\gamma_3(t,x)X(t,x)+\gamma_2(t,x)\gamma_3(t,x)\overline{X}(t,x)-u(t,x)]p(t,x)\\
&+\gamma_4(t,x)[\gamma_1(t,x)X(t,x)+\gamma_2(t,x)\overline{X}(t,x)]q(t,x)\\
&+\int_{\mathbb{R}_0}\gamma_5(t,x,\zeta)[\gamma_1(t,x)X(t,x)+\gamma_2(t,x)\overline{X}(t,x)]r(t,x,\zeta)\nu(d\zeta),
\end{align*}
where $(p,q,r(\cdot))$ is the unique solution to the following BSPDE:
\begin{equation}\label{exam e}
\begin{cases}
dp(t,x)=&-\Big[\frac{1}{2}\Delta p(t,x)+\left(\gamma_1(t,x)+\gamma_2(t,x){\nabla}^{*}_{S}(t,x)\right)\left(\gamma_3(t,x)p(t,x)+\gamma_4(t,x)q(t,x)\right.\\
&\left.+\int_{\mathbb{R}_0}\gamma_5(t,x,\zeta)r(t,x,\zeta)\nu(d\zeta)\right),\qquad\qquad\qquad\qquad\qquad\quad(t,x)\in[0,T]\times D;\\
p(t,x)=&\frac{k(x)}{X(T,x)},\qquad\qquad\qquad\qquad\qquad\qquad\,\qquad\qquad\qquad\qquad\quad\:\: (t,x)\in[T,T+\delta]\times \overline{D};\\
p(t,x)=&0, \qquad\qquad\qquad\qquad\qquad\quad\;\;\qquad\qquad\qquad\qquad\qquad\: \qquad (t,x)\in[0,T)\times \partial D;\\
q(t,x)=&0, \qquad\qquad\qquad\qquad\qquad\quad\;\;\qquad\qquad\qquad\qquad\qquad\: \qquad (t,x)\in[T,T+\delta]\times \overline{D};\\
r(t,x,\cdot)=&0, \qquad\qquad\qquad\qquad\qquad\quad\qquad\qquad\qquad\qquad\qquad\quad\;\;\;\quad (t,x)\in[T,T+\delta]\times \overline{D}.
\end{cases}
\end{equation}
Here ${\nabla}^{*}_{S}$ has been given in Example \ref{exam1}. By Theorems \ref{sufficient} and \ref{thm nece}, the optimal control $\widehat{u}$ satisfies
$$
\widehat{u}(t,x)=\frac{1}{\widehat{p}(t,x)},
$$
where $(\widehat{p},\widehat{q},\widehat{r}(\cdot))$ is the unique solution to \eqref{exam e} for $u=\widehat{u}$ and $X=\widehat{X}$.
\end{example}

\begin{example}\label{+10}
We modify \eqref{per func} to the following performance functional:
\begin{equation}\label{per func2}
J_1(u)=\mathbb{E}\left[\frac{1}{\beta}\int_D\int_0^Tu^{\beta}(t,x)dtdx+\int_Dk(x)X(T,x)dx\right].
\end{equation}
Here, $\beta \in(0,1)$ is a constant and $k(x)\in H_{T}$ is a nonnegative, $\mathscr{F}_T$-measurable process. Now, we aim to find $\widehat{u}(t,x)\in \mathcal{U}^{ad}$ such that
$$
J_1(\widehat{u})=\sup \limits_{u\in\mathcal{U}^{ad}}J_1(u).
$$
The associated Hamiltonian functional in this example becomes the following form
\begin{align*}
H(t,x,S,z,u,p,q,r(\cdot))=&\frac{1}{\beta} u^{\beta}(t,x)+[\gamma_1(t,x)\gamma_3(t,x)X(t,x)+\gamma_2(t,x)\gamma_3(t,x)\overline{X}(t,x)-u(t,x)]p(t,x)\\
&+\gamma_4(t,x)[\gamma_1(t,x)X(t,x)+\gamma_2(t,x)\overline{X}(t,x)]q(t,x)\\
&+\int_{\mathbb{R}_0}\gamma_5(t,x,\zeta)[\gamma_1(t,x)X(t,x)+\gamma_2(t,x)\overline{X}(t,x)]r(t,x,\zeta)\nu(d\zeta),
\end{align*}
where $(p,q,r(\cdot))$ is the unique solution to the following BSPDE:
\begin{equation}\label{exam eq2}
\begin{cases}
dp(t,x)=&-\Big[\frac{1}{2}\Delta p(t,x)+\left(\gamma_1(t,x)+\gamma_2(t,x){\nabla}^{*}_{S}(t,x)\right)\left(\gamma_3(t,x)p(t,x)+\gamma_4(t,x)q(t,x)\right.\\
&\left.+\int_{\mathbb{R}_0}\gamma_5(t,x,\zeta)r(t,x,\zeta)\nu(d\zeta)\right),\qquad\qquad\qquad\qquad\qquad\quad(t,x)\in[0,T]\times D;\\
p(t,x)&=k(x),\qquad\qquad\qquad\qquad\qquad\qquad\,\qquad\qquad\qquad\qquad\quad\, (t,x)\in[T,T+\delta]\times \overline{D};\\
p(t,x)&=0, \qquad\qquad\qquad\qquad\qquad\quad\;\;\qquad\qquad\qquad\qquad\qquad\quad\,  (t,x)\in[0,T)\times \partial D;\\
q(t,x)&=0, \qquad\qquad\qquad\qquad\qquad\quad\;\;\qquad\qquad\qquad\qquad\qquad\quad\,  (t,x)\in[T,T+\delta]\times \overline{D};\\
r(t,x,\cdot)&=0, \qquad\qquad\qquad\qquad\qquad\quad\;\;\qquad\qquad\qquad\qquad\qquad\quad\,  (t,x)\in[T,T+\delta]\times \overline{D}.
\end{cases}
\end{equation}
Here ${\nabla}^{*}_{S}$ has been given in Example \ref{exam1}. By Theorems \ref{sufficient} and \ref{thm nece}, the optimal control $\widehat{u}$ satisfies
$$
\widehat{u}(t,x)=(\widehat{p}(t,x))^{\frac{1}{\beta-1}},
$$
where $(\widehat{p},\widehat{q},\widehat{r}(\cdot))$ is the unique solution to \eqref{exam eq2} for $u=\widehat{u}$ and $X=\widehat{X}$.
\end{example}

\begin{remark}
\begin{enumerate}[($\romannumeral1$)]\mbox{}
\item If we take $\gamma_1(t,x)=1$, $\gamma_2(t,x)=\gamma_5(t,x,\zeta)=0$ $\gamma_3(t,x)=\gamma_3$, $\gamma_4(t,x)=\gamma_4$ and $k(x)=k>0$ in \eqref{exam max}, Example \ref{+10} reduces to Example 3.1 in \cite{Oksendal2005optimal}
\item If we take $\delta=0$, $\gamma_1(t,x)=0$, $\gamma_i(t,x)=\gamma_i\;(i=2,3,4)$, $\gamma_5(t,x,\zeta)=\gamma_5(\zeta)$, $\overline{X}(t,x)=S_1({X}(t,x))$ and $Q_1(x,y)=\frac{1}{V(R_{\theta})}$ in \eqref{exam max}, where $S_1$, $Q_1$ are represented in Example \ref{example}, Example \ref{+10} reduces to Optimal Harvesting (II) in \cite{agram2019spdes}. In addition, if $k(x)=1$, then Example \ref{+9} reduces to Optimal Harvesting (I) in \cite{agram2019spdes}.
\end{enumerate}
\end{remark}



\begin{thebibliography}{10}

\bibitem{agram2019spdes}
Agram, N., Hilbert, A. $\&$ {\O}ksendal, B.
\newblock SPDEs with space-mean dynamics.
\newblock {\em arXiv preprint arXiv:1807.07303}, 2019.

\bibitem{Agram2019Singular}
Agram, N., Hilbert, A. $\&$ {\O}ksendal, B.
\newblock Singular control of SPDEs with space-mean dynamics.
\newblock {\em 	Mathematical Control and Related Fields}, DOI: 10.3934/mcrf.2020004.

\bibitem{agram2019stochastic}
Agram, N. $\&$ {\O}ksendal, B.
\newblock Stochastic control of memory mean-field processes.
\newblock {\em Applied Mathematics $\&$  Optimization}, 79(1): 181-204, 2019.

\bibitem{basse2018multivariate}
Basse-O'Connor, A., Nielsen, M. S., Pedersen, J. $\&$ Rohde, V.
\newblock Multivariate stochastic delay differential equations and CAR representations of CARMA processes.
\newblock {\em Stochastic Processes and their Applications}, 129(10):4119--4143, 2019.

\bibitem{bensoussan2004stochastic}
Bensoussan, A.
\newblock Stochastic Control of Partially Observable Systems.
\newblock {\em Cambridge University Press}, Cambridge, 2004.

\bibitem{Da2014Stochastic}
Da Prato, G. $\&$ Zabczyk, J.
\newblock  Stochastic Equations in Infinite Dimensions.
\newblock {\em  Cambridge University Press}, Cambridge, 2014.

\bibitem{Dumitrescu2018Stochastic}
Dumitrescu, R., {\O}ksendal, B. $\&$ Sulem, A.
\newblock  Stochastic control for mean-field stochastic partial differential equations with jumps.
\newblock {\em  Journal of Optimization Theory and Applications}, 176(3): 559-584, 2018.

\bibitem{Evans2013Stochastic}
Evans, S. N., Ralph, P. L., Schreibe, S. J. $\&$ Sen, A.
\newblock Stochastic population growth in spatially heterogeneous environments.
\newblock {\em Journal of Mathematical Biology}, 66(3):423--476, 2013.

\bibitem{Fuhrman2016Stochastic}
Fuhrman, M. $\&$ Orrieri, C.
\newblock Stochastic maximum principle for optimal control of a class of nonlinear SPDEs with dissipative drift.
\newblock {\em SIAM Journal on Control and Optimization}, 54(1): 341-371, 2016.

\bibitem{Gopalsamy2013Stability}
Gopalsamy, K.
\newblock Stability and Oscillations in Delay Differential Equations of Population Dynamics.
\newblock {\em Springer Science $\&$ Business Media},  Dordrecht, 1992.

\bibitem{hening2018stochastic}
Hening, A., Nguyen, D. H. $\&$ Yin, G.
\newblock Stochastic population growth in spatially heterogeneous environments:
  the density-dependent case.
\newblock {\em Journal of Mathematical Biology}, 76(3):697--754, 2018.

\bibitem{holden1996stochastic}
Holden, H., {\O}ksendal, B., Ub{\o}e, J. $\&$ Zhang, T.
\newblock Stochastic Partial Differential Equations: A Modeling, White Noise Approach.
\newblock {\em  Birkh\"{a}user}, Boston, 1996.

\bibitem{Hu1990Maximum}
Hu, Y. $\&$ Peng, S.
\newblock Maximum principle for semilinear stochastic evolution control systems.
\newblock {\em  Stochastics and Stochastic Reports}, 33(3-4):159-180, 1990.


\bibitem{Kocic2010Generalized}
Kocic, V. L.
\newblock Generalized attenuant cycles in some discrete periodically forced delay population models.
\newblock {\em Journal of Difference Equations and Applications}, 2010, 16(10): 1141-1149.

\bibitem{liu2016analysis}
Liu, J. $\&$ Tudor, C. A.
\newblock Analysis of the density of the solution to a semilinear SPDE with fractional noise.
\newblock {\em Stochastics-An International Journal of Probability and Stochastic Processes}, 88(7):959--979, 2016.

\bibitem{lu2015Stochastic}
L\"{u}, Q.
\newblock Stochastic well-posed systems and well-posedness of some stochastic partial differential equations with boundary control and observation.
\newblock {\em SIAM Journal on Control and Optimization}, 53(6): 3457-3482, 2015.

\bibitem{ma1997Adapted}
Ma, J. $\&$ Yong, J.
\newblock Adapted solution of a degenerate backward SPDE with applications.
\newblock {\em Stochastic Processes and Their Applications}, 70(1): 59-84, 1997.

\bibitem{meng2015optimal}
Meng, Q. $\&$ Shen, Y.
\newblock Optimal control of mean-field jump-diffusion systems with delay: A stochastic maximum principle approach.
\newblock {\em Journal of Computational and Applied Mathematics}, 279(11):13--30, 2015.

\bibitem{mijena2016intermittence}
Mijena, J. B. $\&$ Nane, E.
\newblock Intermittence and space-time fractional stochastic partial differential equations.
\newblock {\em Potential Analysis}, 44(2), 295-312, 2016.

\bibitem{Mohammed1998Stochastic}
Mohammed, S. E. A.
\newblock Stochastic differential systems with memory: theory, examples and applications,  in Decreusefond L., {\O}ksendal B., Gjerde J. $\&$ \"{U}st\"{u}nel A.S. (eds.), Stochastic Analysis and Related Topics VI, pp. 1--77,
\newblock {\em Birkh\"{a}user}, Boston, 1998.

\bibitem{Mokkedem2019Optimal}
Mokkedem, F. Z. $\&$ Fu, X.
\newblock Optimal control problems for a semilinear evolution system with infinite delay.
\newblock {\em Applied Mathematics $\&$ Optimization}, 79(1): 41-67, 2019.

\bibitem{Oksendal2005optimal}
{\O}ksendal, B.
\newblock Optimal control of stochastic partial differential equations.
\newblock {\em Stochastic Analysis and Applications}, 23(1): 165-179, 2005.

\bibitem{Oksendal2011optimal}
{\O}ksendal, B., Sulem, A. $\&$ Zhang T.
\newblock Optimal control of stochastic delay equations and time-advanced backward stochastic differential equations.
\newblock {\em Advances in Applied Probability}, 43(2):572-596, 2011.

\bibitem{oksendal2012Optimal}
{\O}ksendal, B., Sulem, A. $\&$ Zhang T.
\newblock Optimal partial information control of SPDEs with delay and time-advanced backward SPDEs, in Zhang, T. $\&$ Zhou X. (eds.),
\newblock {\em Stochastic Analysis and Applications to Finance: Essays in Honour of Jia-an Yan}, pp.355--383, World Scientific Publishing, Singapore, 2012.

\bibitem{Schreiber2009Invasion}
Schreiber, S. J. $\&$ Lloyd-Smith, J. O.
\newblock Invasion dynamics in spatially heterogeneous environments.
\newblock {\em American Naturalist}, 174(4):490--505, 2009.

\bibitem{Wloka1987Partial}
Wloka, J.
\newblock  Partial Differential Equations.
\newblock {\em  Cambridge University Press}, Cambridge, 1987.

\bibitem{Wu2019Boundary}
Wu, H. N. $\&$ Zhang, X. M.
\newblock  Boundary static output feedback control for nonlinear stochastic parabolic partial differential systems via fuzzy-model-based approach.
\newblock {\em  IEEE Transactions on Fuzzy Systems}, DOI: 10.1109/TFUZZ.2019.2941698.

\end{thebibliography}
\end{document}